\documentclass{amsart}

\usepackage{amsmath, amsthm, amssymb, tikz, hyperref}

\newtheorem{thm}{Theorem}[section]
\newtheorem{cnj}[thm]{Conjecture}
\newtheorem{lem}[thm]{Lemma}
\newtheorem{cor}[thm]{Corollary}
\newtheorem{prop}[thm]{Proposition}

\theoremstyle{definition}
\newtheorem{dfn}[thm]{Definition}
\newtheorem{prb}[thm]{Problem}
\theoremstyle{remark}

\newcommand{\R}{\mathbb{R}}
\newcommand{\N}{\mathbb{N}}

\newcommand{\bb}[1]{\mathbb{{#1}}}
\newcommand{\inv}{^{-1}}
\newcommand{\ip}[1]{\bigl\langle {#1} \bigr\rangle} 
\newcommand{\sm}{\smallsetminus}
 
\newcommand{\dom}{\operatorname{dom}}

\newcommand{\aut}{\operatorname{Aut}}

\newcommand{\res}{\upharpoonright}

\renewcommand{\deg}{\operatorname{deg}}
\newcommand{\br}{\operatorname{br}}

\renewcommand{\bf}{\mathbf}

\newcommand{\casc}{\operatorname{casc}}
\newcommand{\cay}{\operatorname{Cay}}
\newcommand{\rd}{\operatorname{red}}

\title{Factor of iid colorings of trees}
\author{Riley Thornton}
\address{CMU Math Department \\ 5000 Forbes Ave \\ Pittsburgh PA, 15232}
\email{rthornto@andrew.cmu.edu}

\begin{document}

\maketitle
\begin{abstract}
    We show that, for every $\epsilon>0$, the 4-regular tree has an fiid 4-coloring where a given vertex is assigned the 4th color with probability at most $\epsilon$. We also construct 5-colorings of $T_6$ improving known bounds on the measurable and approximate chromatic number of $F_3$.
\end{abstract}

\section{Introduction}
In this note, we'll adapt an argument of Achlioptas and Moore \cite{AM} to make progress on a long-standing problem: what is the measurable chromatic number of the Bernoulli shift of $F_2$? (See, for example, the survey by Kechris and Marks \cite{KM} for more on this problem.) In particular, we compute the approximate chromatic number of $F_2$. A well known equivalence between measurable labelling problems and factor of iid processes tells us that answering this question amounts to finding the least $k$ for which $T_4$ admit an (approximate) $F_2$-fiid $k$-coloring. We in fact prove:
\begin{thm} \label{thm:mainthm}
    For any $\epsilon>0$, there is an $\aut(T_4)$-fiid $4$-coloring of $T_4$, $\bf f$, so that for any vertex $v$ 
    \[\bb P\bigl(\bf f (v)=3 \bigr)<\epsilon.\]
\end{thm} An approximate 2-coloring is impossible by strong ergodicity, so the approximate chromatic number of $ F_2$ is $3$. We also give an fiid $5$-coloring of $T_6$ which improves the known bounds on the measurable (not merely approximate) chromatic number of $F_3$. 
\begin{thm} \label{thm:submainthm}
    For any $\epsilon>0$ there is an $\aut(T_6)$-fiid $5$-coloring of $T_6$ so that a vertex gets color $4$ with probability at most $\epsilon$.
\end{thm}

The method also yields $5$-colorings of $T_5$ which are approximate $4$-colorings, but this also follows from the theorem of above. Before beginning, we warn the reader that the proofs given here rely on machine computation, but our computation can be verified on most desktop computers.

\subsection{Acknowledgements} I would like to thank Andrew Marks and Clinton Conley for pointing out the literature on the differential equation method to me. I would like to thank Colin Jahel for a number of helpful conversations. And, I would like to thank MATLAB for verifying Proposition \ref{prop:comp}. This work was partially supported by NSF MSPRF grant DMS-2202827.

\subsection{Notation and conventions}

A graph is a pair $G=(V,E)$ where $E\subseteq V^2$ is symmetric and irreflexive. We write $u\sim v$ to mean that $(u,v)\in E$, and say that $u,v$ are neighbors or are adjacent. A connected graph induces a metric on its vertex set by counting the number of edges in the smallest path between two vertices. We write $B_n(v)$ to mean the closed ball of radius $n$ in this metric:
\[B_n(v)=\{u: d(u,v)\leq n\}.\] And, we write $N(v)$ for the set of neighbors of $v$:
\[N(v)=\{u: u\sim v\}.\] So, $N(v)=B_1(v)\sm\{v\}$. We write $T_d$ for the (unique-up-to-isomorphism) $d$-regular tree.

We will work with partial functions, $f:X\rightharpoonup Y$. Formally, we view a partial function on $X$ as a function $f:A\rightarrow Y$ with $A\subseteq X$, and we call $A$ the domain of $f$, or $\dom(f)$. When we write $f(x)=y$ for a partial function $x$, we mean to imply that $x\in \dom(f)$.

We reserve bold characters for random variables, and the domains of our random variables will all be product spaces of the form $([0,1]^V, \lambda^V)$ where $\lambda$ is Lebesgue measure and $G=(V, E)$ is a countable graph. Of course $([0,1],\lambda)$ is isomorphic to any standard probability space, such as $([0,1]^\N,\lambda^\N)$. So, in practice, we assume every vertex of our graph can independently sample any source of randomness we require infinitely often.

Usually, our random variables will take values in $A^V$ for some $A$ and will be $\aut(G)$ equivariant; this is what is meant by an fiid labelling of $G$. More generally, $\bf f$ is $\Gamma$-fiid if $\Gamma$ acts on both $V$ (and by extension $[0,1]^V$) and the codomain of $\bf f$, and for all $\gamma\in \Gamma$ \[\bf f(\gamma\cdot x)=\gamma \cdot \bf f(x).\] We suppress the dependent variable in our notation. So, we write $\bb P\bigl(\phi(\bf f)\bigr)$ for the probability that $\bf f$ has property $\phi$, e.g. for a vertex $v\in V$
\[\bb P\bigl(\bf f(v)=3\bigr)=\lambda^V\bigl(\bigl\{x:\bigl(\bf f(x)\bigr)(v)=3\bigr\}\bigr)\]

Similarly $\bb E\bigl(g(\bf f)\bigr)$ is the expected value of $g(\bf f)$, e.g.
\[\bb E\bigl(\# \bigl\{ u\in B_3(v):\bf  f(u)=3\bigr\}\bigr)=\int_x \# \bigl\{ u\in B_3(v): \bigl(\bf f(x)\bigr)(u)=3\bigr\}.\]

We are concerned with (partial) colorings of $T_n$. Say a (partial) coloring $f$ is proper at $v$ if $f(v)\not=f(u)$ for any neighbor $u$ of $v$. And, say $f$ is proper if it is total and proper at every vertex. We will allow (at first) our colorings to be improper at vertices in a small error set. We will reserve the color red for vertices where we have made an error. Write $k$ for the von Neumann numeral $\{0,...,k-1\}$ and set 
\[k^+=\{0,...,k-1,\rd\}.\] For a partial coloring $f$ and vertex $v$, say a color $i\in k$ is available to $v$ if $f(u)\not=i$ for any $u\sim v$. Otherwise, say $v$ sees $i$. Note, when we say a color is available to a vertex we always mean a color other than red. 

By an $n$-coloring we mean a proper coloring with $n$ colors. When we want to emphasize that a coloring may not be proper we call it a labelling. An fiid approximate $n$-coloring of $G=(V,E)$ is a sequence of fiid labellings $\ip{\bf f_i: i\in\omega}$ of $G$ so that $\bb P(\bf f_i(v)\in n),$ $\bb P(\bf f_i\mbox{ is proper at }v)$ both converge to $1$.

\subsection{A note about translation} This article is aimed primarily at descriptive set theorists, though it is couched in probabilistic language. In many cases there is an easy translation.

There is an exact correspondence between $F_n$-fiid labellings of $T_{2n}=\cay(F_n)$ and measurable labellings of the Schreier graphing of the $F_n$-Bernoulli shift (one can find other groups corresponding to $T_d$ in general). An fiid labelling \[\bf f:[0,1]^F_n\rightarrow k^{T_{2n}}\] corresponds to a measurable labelling via
\[f(x)=\bf f(x)(e)\] (where $e$ is the root of $T_{2n}$). The inverse is given by
\[\bf f(x)(\gamma)= f(\gamma\cdot x).\] (Or $f(\gamma\inv\cdot x)$ depending on your conventions for the shift action and the Cayley graph). So, for example $\bb P(\bf f(v)=7)=\mu(\{x: f(x)=7\}).$ More generally, the probability of some event is the measure of the vertices which see this event locally. 

This paper concerns not just $F_{n}$-fiid processes, but $\aut(T_d)$-fiid processes. If we set $R=\{\rho\in\aut(T_{2n}): \rho(e)=e)\}$ (again, $e$ is a root of $T_{2n}$), then $\bf f$ is $\aut(T_{2n})$-equivariant if and only if it is $R$-equivariant and $F_{n}$-equivariant. So, $\aut(T_{2n})$-fiid processes correspond to labellings of the Bernoulli shift with extra symmetry. In fact, since $R$ is compact, we can quotient out by these symmetries to get a pmp graph which models $\aut(T_{2n})$-factor maps.

I will point out the translations to descriptive set theoretic language in a couple places, but several important arguments here are simply more natural in the probabilistic language. 

From this point forward, by fiid we mean $\aut(T_n)$-fiid.

\section{Approximate colorings}

We'll first show that $T_4$ has an fiid approximate 3-coloring. That is, we will show that there is an fiid labelling of $T_4$ so that a given vertex gets colored $\{0,1,2\}$ with probability $(1-\epsilon)$ and the labelling is proper at a given vertex with probability $(1-\epsilon)$. And similarly, $T_6$ has an fiid approximate $4$-coloring. In later sections we will tidy up these colorings to get fiid proper $4$- and $5$-colorings with one color class as small as we like.

\subsection{Overview}

 It will be helpful to have an outline of the argument in mind. For this discussion we will analyze an attempt to 3-color $T_4$ (though the same discussion applies almost verbatim to other cases). The algorithm runs in two phases as follows: In each step of the first phase, we greedily attempt to color some small fraction (say $\epsilon$) of the remaining vertices so that every vertex remaining has at least 2 colors available to it. We activate each vertex iid at random with probability $\epsilon$ (weighted according to some local data) and color the active vertices. Coloring a vertex may rob some neighbor and reduce its list of available colors to one, so we are forced to color that neighbor. Coloring this new vertex will rob others, and so on. We chase this cascade to its end and hope no cascade grows too large. Of course, two cascades may crash into each other and ruin our coloring. For now, color the vertices where two clashes meet red (the color of error). Provided the probability of seeing a cascade of size $n$ decays exponentially in $n$, these clashes occur with probability $O(\epsilon^2)$. Since in each step we color a vertex red much less often than we color it any color, we will end up with an approximate coloring. The first phase ends when the expected size of a component in the remainder graph is finite. Then in the second phase we simply choose a coloring on each component (using, for example, the Luzin--Novikov measurable uniformization theorem). Such a coloring is available because trees are 2-list-colorable. In the end we may have some adjacent red vertices; we will see how to tidy these up in a later section.

 For this algorithm to work, we need that the sizes of cascades stay small (i.e. decay exponentially) until the components of the remainder become finite. In the following subsection we will prove a handful of lemmas telling us the cascades and the components in the remainder behave like branching processes, so we can compute their expected sizes rather easily; they only depend on the probability that a vertex in the uncolored remainder sees a given (multi)set of colors. We can also compute the change in these probabilities in each step of the first phase. It turns out the probabilities satisfy (up to small error) a difference equation which looks like an Euler approximation (with step size $\epsilon$) to a certain differential equation. Using numerical computation software we can check that, taking $\epsilon$ small enough, the branching process describing the cascades stays subcritical until the branching process describing components of the remainder becomes subcritical. We derive these differential equations and report the computation in Section \ref{sec:approx}.

\subsection{Definition and some lemmas}

In this section, we formally describing the first phase of our algorithm and prove some important lemmas about its behaviour. Throughout we will think of the vertices colored red as forming a kind of error set. Our set up will be phrased in somewhat clunky generality; we attempt to $p$-color $T_r$ ($p$ for palette, $r$ for regularity).

\begin{dfn} \label{dfn:greedycolor}
    If $\bf f$ is a random partial coloring of $T_r$, we call the (random) set of uncolored vertices the \textbf{remainder graph} of $\bf f$ and write $R_{\bf f}$. We say that vertices in this set \textbf{remain}.

    Fix $\epsilon>0$ (which will be a step size) and tuning parameters\footnote{To keep cascades small, we'll prefer to activate vertices with few uncolored neighbors} $p_{(d,c)}<1$ for $(d,c)$ in the set 
    \[T=\{(d,c): 0\leq d\leq r, 2\leq c\leq p\}\] ($c$ will be the number of colors available to a vertex and $d$ will be the number of uncolored neighbors $v$ has). Then, the \textbf{greedy coloring process} for $T_r$ with these parameters is the sequence of fiid partial $p^+$-colorings of $T_r$, $\ip{\bf{f}_i:i\in\N}$, defined inductively:
    \[\mathbf{f}_0=\emptyset\]
    and $\mathbf{f}_{i+1}$ is gotten from $\bf f_i$ as follows: first consider a set $A_i$ of ``active" vertices, with each $v\in A$ with probability $\epsilon p_{(d,c)}$ conditionally independently given that $v$ remains, has $d$ uncolored neighbors, and has $c$ colors from $p$ available. Iteratively apply the following rules to each vertex $v$ simultaneously in rounds
    \begin{enumerate}
        \item if $v$ is in $A_i$, color $v$ by a random color available to $v$
        \item if $v$ has $2$ colors available, say $j$ and $k$ and one of $v$'s neighbors is colored $k$, then color $v$ with $j$.
        \item if two or more of $v$'s neighbors have been colored in this step (possibly in different rounds), color $v$ red
        \item if $2$ neighbors $v,u$ would be colored simultaneously, then color them both red.
    \end{enumerate}
    Then, $\bf f_{i+1}$ is the coloring we have after applying the above rules in $\omega$ rounds.
\end{dfn}

Note that, if $\bf f_i(u)$ is not red, then $\bf f_i$ is proper at $u$, meaning that no neighbor of $u$ has the same color as $u$. And for any $i$, any remaining vertex in $\bf f_i$ has at least two colors available from $p$.

It will also be useful to analyze this coloring process starting from a single vertex in isolation:

\begin{dfn}
    Suppose $\bf f$ is a partial $p^+$-labelling of $T_4$ with the property that every remaining vertex sees at most one color from $p$. For $v$ a vertex of $T_r$, the cascade at $v$ is the random partial coloring defined inductively as follows:
    \begin{itemize}
    \item If $v$ is colored by $\mathbf{f}$, then $\casc_0(v,\bf{f})$ is empty. Otherwise, $\casc_0(v,\bf f)$ colors $v$ with a random color available to it (chosen uniformly at random, independent from the rest of the process)
    \item If $u$ is adjacent to a vertex colored by $\casc_i(v,\bf{f})$, and $u$ is not colored by $\bf f$, then $\casc_{i+1}(v,\bf{f})$ colors $u$ if $u$ now has only one color from $p$ available, and $u$ gets whichever color is available.
    \end{itemize}
    And $\casc(v,\bf f)=\bigcup_i \casc_i(v,\bf f).$ Note that $|\casc(v,\bf f)|=|\dom(\casc(v,\bf f)|$. We'll write $\casc(v, i)$ for $\casc(v, \bf f_i)$.
\end{dfn}

Note that if $v$ is active in step $i$ of the greedy coloring process, then every vertex colored by $\casc(v, \bf f_i)$ is colored by $\bf f_{i+1}$ or separated from $v$ by a vertex colored red by $\bf f_{i+1}$. Conversely, everything colored by $\bf f_{i+1}$ is either colored by the cascade at some active vertex in $\bf f_i$ or is colored red.

In the first phase, the algorithm has the feature that colors only spread outward from active vertices. Labellings coming from such operations have the useful feature that the behavior of the labelling is independent on different branches of the remainder.

\begin{dfn}
Given $u,v$ neighbors in $T_r$, the \textbf{branch} of $u$ from $v$ is the component of $u$ after deleting $v$. We write
\[\br_v(u):=\mbox{ the connected component of }u\mbox{ in }T_r\sm \{v\}\]

A random partial coloring $\mathbf f$ of $T_r$ is \textbf{nearsighted} if $\bigl(\bf f\res \br_v(u)\bigr)$ and $\bigl(\bf f\res \br_u(v)\bigr)$ are conditionally independent given that $u,v$ are uncolored.
\end{dfn}
In the language of pmp graphs, if $\bf f$ is nearsighted then in the subgraph of vertices uncolored by $\bf f$ with normalized induced measure any sets defined by looking at different branches of the connected component of a vertex will be independent. The proof of the next theorem involves a non-equivariant labelling and does not have a very natural translation to the language of pmp graphs.

\begin{lem}
For any choice of parameters, and for each $i$, the greedy coloring process at step $i$, $\bf f_i$, is nearsighted.
\end{lem}
\begin{proof}
For $u,v$ vertices of $T_r$, consider the modified (not equivariant) coloring $\bf{\tilde f}$ defined the same as $\bf f_i$ except that any time a cascade starting in $\br_v(u)$ reaches $u$ it stops, and likewise for $\br_u(v)$; that is, we amend Definition \ref{dfn:greedycolor} parts $(1-4)$ to only concern neighbors in the same branch.

Then, $u,v$ remain in $\bf f_i$ if and only if they remain in $\bf{\tilde f}$. And, given $u,v$ remain in either, $\bf f_i=\bf{\tilde f}$.

For brevity, we'll write $f_{\ell}=f\res \br_u(v)$ for a partial coloring $f$ of $T_4$. We call $f_\ell$ the left branch of $f$. Similarly, we write  $f_r=f\res \br_v(u)$ for the right branch of $f$.

 Since any events concerning $\bf{\tilde f}_\ell$ and $\bf{\tilde f}_r$ are independent, we can compute, for $A$ a property of left branches of colorings and $B$ a property of right branches:

\begin{align*}
\bb P\bigl(A\bigl((\bf f_i)_\ell\bigr) \wedge B\bigl((\bf f_i)_r\bigr)|\;u,v\in R_{\bf{f}_i}\bigr) & = \bb P\bigl(A\bigl(\bf{\tilde f}_\ell\bigr)\wedge B\bigl(\bf{\tilde f}_r\bigr)|\;u,v\in R_{\bf{\tilde f}}\bigr)  \\ 
\; & = \bb P\bigl(A\bigl(\bf{\tilde f}_\ell\bigr) |\;u,v\in R_{\bf{\tilde f}}\bigr)\bb P\bigl( B\bigl(\bf{\tilde f}_r\bigr)|\;u,v\in R_{\bf{\tilde f}}\bigr) \\
\; & = \bb P\bigl(A\bigl((\bf{f}_i)_\ell\bigr)|\;u,v\in R_{\bf{ f}_i}\bigr)\bb P\bigl( B\bigl((\bf{ f}_i)_r\bigr)|\;u,v\in R_{\bf{f}_i}\bigr) 
\end{align*}
Where in the second line we use that $u\in R_{\bf{\tilde f}}$ and $v\in R_{\bf{\tilde f}}$ are independent of each other and of events concerning the opposite branch.
\end{proof}

So, as we start coloring vertices in a cascade, the probability that coloring a vertex robs one of its neighbors of a color doesn't depend on the past of the cascade. The cascades behave like multi-type branching processes. The next few lemmas make this precise and give the parameters for the branching process. (See, for example, \cite[Chapter 5]{LP} for background on branching processes.)

\begin{dfn}
If $F$ is a partial labelling of $T_r$, we write $\deg_F(v)$ for the number of unlabelled neighbors of $v$, and $c_F(v)$ for the set of colors from $p$ available to $v$:
\[c_f(v)=p\sm \{F(u):u\sim v\}.\]

The \textbf{type} of an uncolored vertex $v$ in a partial coloring $F$ is the pair \[\tau_F(v):=\bigl(\deg_F(v), |c_F(v)|\bigr).\] We suppress $F$ when it is clear from context, and we write $\tau_i(v)$ for $\tau_{\bf f_i}(v)$.

Let us stress that types are only defined for uncolored vertices: $v\in \dom(\tau_i)$ implies $v\in R_{\bf f_i}$.  
\end{dfn}

Recall that $T=\{(d,c): 0\leq d\leq r, 2\leq c\leq p\},$ so $\tau_i(v)\in T$ for all $i,v$. Many types in $T$ will not appear (for instance, $(r,1)$ is impossible), and we'll see that even more types will only show up with negligible probability. This will only matter in making computation more tractable.

Note that the greedy coloring process is symmetric under permuting the colors in $p$, so if $\tau_i(v)=(d,c)$ then the probability that $c_i(v)=C$ is the same for all $C\subseteq p$ with $|C|=c$.

To understand cascades in the greedy coloring process, we just need to understand the probability of finding a vertex of a given type when we take a step from an uncolored vertex. We compute these probabilities in the next lemma.

\begin{lem}\label{lem:step} For any fiid $\bf f$, $u, v$ neighbors in $T_r$, and type $t$ with $\deg(t)=d$, the probability $\bb P\bigl(\tau_{\bf{f}}(v)=t\,|\, u,v\in R_{\bf{f}}\bigr)$ is proportional to $d\,\bb P\bigl(\tau_{\bf{f}}(v)=t\bigr)$. In particular,
\[\bb P\bigl(\tau_{\bf{f}}(v)=t\, | \, u\in R_{\bf{f}}\bigr)=\frac{d\,\bb P\bigl(\tau_{\bf f}(v)=t\bigr)}{\sum_{s\in T}\deg(s) \bb P\bigl(\tau_{\bf{f}}(v)=s\bigr)}.\]
\end{lem}
\begin{proof}
It suffices to show the same proportionality for $\bb P(\tau_{\bf f}(v)=t\wedge u\in R_{\bf f})$ (since $v\in \dom(\tau_i)$ implies $v\in R_{\bf f_i}$). We compute:
\begin{align*} \bb P(\tau_{\bf f}(v)=t \wedge u\in R_{\bf f}) & = \sum_{\substack{A\subseteq N(v) \\ u\in A}}{\bb P\bigl(A=(N(v)\cap R_{\bf f})\wedge \tau_{\bf f}(v)=t}\bigr)
\end{align*}
By equivariance, the summand only depends on $|A|$, and is $0$ unless $|A|=d$. So,
\begin{align*}\bb P(\tau_{\bf f}(v)  =t \wedge u\in R_{\bf f})&=\frac{\binom{r-1}{d-1}}{\binom{r}{d}}\, \bb P(\tau_{\bf f}(v)=t) \\
\; & = \frac{d}{r}\, \bb P(\tau_{\bf f}(v)=t).\end{align*}
\end{proof} 
The forgoing argument translates naturally to language of pmp graphs. For $x\in [0,1]^V$, write $x\in R_{\bf f}$ to mean that $e\in R_{\bf f(x)}$, and likewise for $\tau_{\bf f}(x)=t$. By $R$-equivariance, for any generator $a$ of $F_n$,
\[\bb P\bigl(\tau_{\bf f}(v)=t\wedge u\in  R_{\bf f}\bigr)=\mu\bigl(\{x\in R_{\bf f}: \tau_{\bf f}(a\cdot x)=t\}\bigr).\] So, summing over all generators for $F_n$:
\[\bb P\bigl(\tau_{\bf f}(v)=t\wedge u\in R_{\bf f}\bigr)=  \frac{1}{2n}\int_{x\in {R_{\bf f}}} \bigl|N(x)\cap \{y: \tau_{\bf f}(y)=t\}\bigr|\;d\mu.\] And by the measure theoretic handshake lemma,
\[\int_{R_{\bf f}} \bigl|N(x)\cap \{y: \tau_{\bf f}(y)=t\}\bigr|\;d\mu=\int_{\{y: \tau_{\bf f}(y)=t\}} \bigl|N(y)\cap R_{\bf f}\bigr|\;d\mu=d\, \bb P\bigl(\tau_{\bf f}(y)=t\bigr).\]

It will be useful to have some notation for these probabilities:
\begin{dfn}
For each $t\in T$, \[q_t(i):=\frac{\deg(t)\bb P\bigl(\tau_i(v)=t\bigr)}{\sum_{s\in T} \deg(s) \bb P\bigl(\tau_i(v)=s\bigr)}.\]
\end{dfn}

Putting together these lemmas, we have
\begin{prop}\label{prop:branching}
    Write $\casc^t_k(v,i)$ for the set of vertices with $\tau_i(v)=t$ in the domain of $\casc(v,i)$ at distance $k$ from $v$. Then, for each $i$, the sequence of random integer vectors
    \[\Bigl\langle{\ip{|\casc^t_k(v,i)|:t\in T}, k\in \N}\Bigr\rangle\] is a multiptype branching process so that each vertex in $\casc^s_{k}(v,i)$ of type $s$ has each of its $(\deg(s)-1)$-many children in $\casc^t_{k+1}(v,i)$ independently with probability $\frac 2 p q_t(i)$ provided $t$ has exactly 2 colors available. That is,
    \[|\casc^t_{k+1}(v,i)|=\sum_{s\in T}\sum_{\ell=1}^{|\casc^s_{k}|} \sum_{j=1}^{\deg(s)-1} \bf X_{j,t}\]
    where $\bf X_{j,t}$ is a sequence of iid $\{0,1\}$-valued random variables with \[\bb P(\bf X_{j,t}=1)=\frac{2}{p}q_t(i)\] for any type $t$ with exactly 2 colors available from $p$ (and $\bf X_{t,j}$ is always $0$ for all other $t$).
\end{prop}

\begin{proof}
    By nearsightedness and equivariance, if we condition on any particular set of size $j$ being $\casc_k^s(v,i)$, we get that each of the $(j\deg(s))$-many children of vertices in this set are in $\casc^t_{k+1}(v,i)$ independently with same probability. And Lemma \ref{lem:step} tells us this probability is given by the formula above, as we see a vertex with type $t$ with probability $q_t(i)$, and we color it if we take one of its colors, i.e. with probability $(2/p)$.
\end{proof}
And similarly, the components of the remainder graph are given by a simple branching process:
\begin{prop}\label{prop:components}
    The size of the component of $u$ in the remainder is a simple branching process. That is, the sequence of random integers
    \[\Bigl\langle\bigl|[u]_{R_{\bf f_i}}\cap \{v: d(u,v)=k\}\bigr|:\;k\in \N\Bigr\rangle\] is a branching process where each vertex has $k$ children with probability \[\sum_{\substack{t\in T \\ \deg(t)=(k+1)}} q_t(i)\]
\end{prop}
\begin{proof}
    Conditioning on $v$ being in $[u]_{R_{\bf f_i}}$, we have that the parent of $v$ is also in this component. So, $v$ has type $(d,c)$ with probability $q_{(d,c)}$. And, $v$ has $k$ children if and only if it has type $(k+1,c)$ for some $c$. Lastly, by nearsightedness, the number of children for $v,v'$ are independent for different $v,v'$ in $[u]_{R_{\bf f_i}}$.
\end{proof}

\subsection{An approximate coloring} \label{sec:approx}

Now we're in position to fill in the details of the overview. For $t$ a type, we write $\deg(t)$ for the left coordinate and $c(t)$ for the right. We want to track the distribution of types throughout the operation of the algorithm.

\begin{dfn}
    For $\vec z=\ip{z_t: t\in T}\in \R^T$, define
    \[q_s(\vec z)=\frac{\deg(s)\, z_s}{\sum_{t\in T} \deg(t)\,z_t},\] and set 
    \[g(\vec z)=\frac{2}{p}\sum_{d\leq r} q_{(d,2)}\] 
    
    Let $\vec z(i)=\ip{z_t(i):t\in T}$ record the distrbution of types in $\bf f_i$:
    \[\quad z_t(i):=\bb P(\tau_{\bf{f}_i(v)}=t ).\] So, in particular, $q_s(\vec z)=q_s(i)$. Write $g(i)$ for $g(\vec z(i))$.
    
    Say that $\vec z$ (or a random partial coloring with type distribution $\vec z$) is \textbf{subcritical} if $g(\vec z)<1$. So, $\bf f_i$ is subcritical if $g(i)<1$.

\end{dfn}

Since we only define types for uncolored vertices, $\bb P(v\in R_{\bf f_i})=\sum_{t\in T}z_t(i)$, and \[\bb P\bigl(\tau_i(v)=t\,|\,v\in R_{\bf f_i}\bigr)=\frac{z_t(i)}{\sum_{s\in T} z_s(i)}.\]

Note that $\vec z(i)$ is not a random variable. It only depends on the tuning parameters and $\epsilon$ chosen for the greedy coloring process. We will show that $\vec z(i)$ closely approximates the solution a certain differential equation (after re-scaling the time step). To do so, we need to understand when cascades stay small.

\begin{prop}
    If $\bf f_i$ is subcritical then the expected number of vertices colored by $\casc(v,i)$ is finite, and the probability that $|\casc(v,i)|=n$ decays exponentially in $n$
\end{prop}
\begin{proof}
    From Proposition \ref{prop:branching}, for each type $s$ with $c(s)=2$, each vertex in $\casc_k(v,i)$ has type $s$ independently with probability $q_s(i)/\sum_{d\leq r} q_{(d,2)}(i)$. And a child of a vertex of type $s$ in $\casc_k(v,i)$ is in $\casc_{k+1}(v,i)$ independently with probability $(2/p)\sum_{d\leq r} q_{(d,2)}(i)$. So
    \[\ip{|\casc_k(v,i)|:k\in\N}\] is a simple branching process with the expected number of children of a vertex given by $\frac{2}p\sum_{d\leq r} q_{(d,2)}(i)(d-1)$. The rest follows by standard facts about branching processes.
\end{proof}

We will now derive the differential equations which describe the trajectory of our algorithm. Then all that's left is integrate the equation and tidy up some loose ends.

\begin{dfn}
    For $s,t\in T$, we write
    \begin{align*}\Delta_{s,t}(i):= & \bb E\bigl(\#\;\mbox{vertices of type }t\mbox{ made by }\casc(v,i)\,|\,\tau_i(v)=s\bigr) \\
     =&\bb E\bigl(\# \;\mbox{vertices with type changed to }t\mbox{ by }\casc\,|\,\tau_i(v)=s\bigr)\\
     & \; -\bb E\bigl(\# \mbox{ vertices with type changed from }t\mbox{ by }\casc\,|\,\tau_i(v)=s\bigr)\\
    =& \bb E\bigl(\bigl|\{u: \tau_i(u)\not=t\wedge \tau_{\bf f_i\cup\casc(v,i)}(u)=t\}\bigr| \,|\,\tau_i(v)=t\bigr)\\
    & \; -\bb E\bigl(\bigl|\{u: \tau_i(u)=t\wedge \tau_{\bf f_i\cup\casc(v,i)}(u)\not=t\}\bigr| \,|\,\tau_i(v)=s \bigr)\end{align*}
\end{dfn}

We will compute $\Delta_{s,i}(i)$ in a moment and see that it depends only on $\vec z(i)$. So, the next lemma gives a difference equation describing the trajectory of $\vec z(i)$. Note that $\Delta_{s,t}(i)$ is finite if and only if $\bf f_i$ is subcritical.

\begin{lem}[Step Lemma]\label{lem:steplemma}
    Provided $\vec z(i)$ is subcritical, for any $t\in T$
    \[z_t(i+1)=z_t(i)+\sum_{s\in T} \epsilon p_s z_s(i)\Delta_{s,t}(i)+O(\epsilon^2) \] where the hidden constant only depends on the parameters $p_s$ and the expected growth of cascades (i.e. $g(i)$).
\end{lem}
\begin{proof}
    First we want to control the probability that a vertex is affected by more than one active cascade.

    To keep our formulae on one line, we need to introduce some notation. Let $C(u,v)$ be the event that $\casc(u,i)$ colors $v$, and let $B$ be the even that $v$ is involved in more than $2$ active cascades. Recall that $A_i$ is the set of vertices active at step $i$. Fix $x\sim v$. Then, considering the cases where $v$ is or is not active we can compute
    \begin{align*}\bb P(B)&\leq \sum_{w\sim v} \sum_{u\in \br_{v}(w)} \bb P\bigl(v, u\in A_i, C(u,v)\bigr)+\binom{r}{2} \left(\sum_{u\in \br_{v}(x)} \bb P\bigl(u\in A_i, C(u,v)\bigr)\right)^2\\
    &\leq r\sum_{u\not=v}\bb P\bigl(v,u\in A_i, C(u,v)\bigr)+\binom{r}{2}\left(\sum_{u\not=v} \bb P\bigl(u\in A_i, C(u,v)\bigr)\right)^2. \end{align*} Let's bound the left and right hand sums separately. In each case we use mass transport (see \cite[Chapter 8]{LP}):
    \begin{align*}\sum_{u\not= v} \bb P\bigl(u\in A_i, C(u,v)\bigr) &  = \sum _{u\not=v}\sum_{t\in T} \bb P\bigl(u\in A_i, C(u,v)\,|\,\tau_i(u)=t\bigr)\bb P\bigl(\tau_i(u)=t\bigr)\\
    &= \sum_{u\not=v}\sum_{t\in T} \epsilon p_t\bb P\bigl(C(u,v)\,|\,\tau_i(u)=t\bigr) \bb P\bigl(\tau_i(u)=t\bigr)\\ 
    & \leq \epsilon \sum_{u\not=v} \sum_{t\in T} \bb P\bigl(C(u,v)\,|\,\tau_i(u)=t\bigr)\bb P\bigl(\tau_i(u)=t\bigr)\\
    &\leq \epsilon \sum_{u\not=v} \bb P\bigl(C(u,v)\bigr)\\
    & = \epsilon \sum_{u\not=v}\bb P\bigl( C(v,u)\bigr)\\
    &\leq \epsilon\, \bb E\bigl(|\casc(v,i)|\bigr).\end{align*} And similarly: 
    \[\sum_{u\not=v} \bb P(u,v\in A_i, C(u,v))\leq \epsilon^2\, \bb E\bigl(|\casc(v,i)|\bigr).\] So, all told \[\bb P(B)\leq 2\left(\epsilon\,r\,\bb E\bigl(|\casc(v,i)|\bigr)\right)^2.\]
    
   Now we want to compute $z_t(i+1)$. From our computation of $B$, we can assume $v$ and its neighbors are colored by at most one cascade. All other cases will be subsumed under an $O(\epsilon^2)$ error. Let $C_+(u,v)$ be the event that $\casc(u,i)$ changes $\tau(v)$ to $t$ from anything else, and let $C_-(u,v)$ be the even that $\casc(v,i)$ changes $\tau(v)$ from $t$ to anything else. Then,
    \begin{align*} z_t(i+1) & =\bb P\bigl(\tau_{i+1}(v)=t\bigr)\\
                   \; &= \bb P\bigl(\tau_i(v)\wedge \tau(v)\mbox{ doesn't change}\bigr) +\bb P\bigl(\tau_i(v)\not=t\wedge \tau(v) \mbox{ changes to }t\bigr)\\
                   & = z_t(i)+\sum_{u\in V}\left( \bb P\bigl(u\in A_i, C_+(u,v)\bigr)-\bb P\bigl(u\in A_i, C_-(u,v)\bigr) \right)+\bb P(E)
    \end{align*}
    Where $E$ is a small error set contained in $B$. And, we can compute the sum
    \[S:=\sum_{u\in V} \left( \bb P\bigl(u\in A_i, C_+(u,v)\bigr)-\bb P\bigl(u\in A_i, C_-(u,v)\bigr)\right)\] by a method similar to the computation above.
    \begin{align*}S&= \sum_{u\in V} \sum_{s\in T}\epsilon\,p_s\,z_s(i)\,\Bigl(\bb P\bigl(C_+(u,v)\,|\,\tau_i(u)=s\bigr)-\bb P\bigl(C_-(u,v)\,|\,\tau_i(u)=s\bigr)\Bigr)\\
    & = \sum_{s\in T}\epsilon \,p_s \,z_s(i)\sum_{u\in V}\Bigl(\bb P\bigl(C_+(v,u)\, | \, \tau_i(v)=s\bigr)-\bb P\bigl(C_-(v,u)\,|\,\tau_i(v)=s\bigr) \Bigr) \\
    & = \epsilon\sum_{s\in T} p_s\,z_s(i)\,\Delta_{s,t}\bigl(\vec z(i)\bigr).\end{align*}
So, combining this all, we get
\[z_t(i+1)=z_t(i)+\sum_{s\in T} p_s\, z_s(i)\, \Delta_{s,t}\bigl(\vec z(i)\bigr)+O(\epsilon^2)\] with implicit constants as desired.
\end{proof}
Note that in $3$-coloring $T_4$, the types $(2,3)$ and $(1,3)$ only occur at vertices adjacent to red vertices. Since we'll red vertices only show up with probability $\epsilon$ throughout the greedy coloring process, the contributions of these types can be folded into the $\epsilon^2$ error. That is, we can disregard these types. Similar considerations apply to $p$-coloring $T_r$ in general.

It's probably more informative for the reader to work out formulas for $\Delta_{s,t}$ on their own steam, but we'll include a brief derivation here. The important facts are that these depend only on $\vec z(i)$, are Lipschitz functions when restricted to any region where $g(\vec z)$ stays bounded away from 1, and can be handled ably by most numerical computing software.

   Recall that \[q_t(i)=\deg(t)\,z_t(i)/\sum_{s\in T}\deg(s)\,z_s(i),\quad \mbox{and} \quad g(i)=(2/p)\sum_{d\leq r}(d-1)q_{(d,2)}(i).\]
\begin{dfn}
    For types $t,(d,c)\in T$, define
    \[\Delta_{t,(d,c)}(\vec z)=-\delta_{t,(d,c)}+\deg(t)\Delta_{B,(d,c)}(\vec z)\]  where $\Delta_{B,(d,c)}$ is
    \[\frac{1}{1-g(\vec z)}\bigl(-q_{(d,c)}(\vec z)+\frac{c+1}{p} q_{(d+1,c+1)}(\vec z)+\frac{p-c}{p}q_{(d+1,c)}(\vec z)\bigr)\] and $q_t(\vec z)=0$ when $t\not\in T$.
\end{dfn}

\vspace{5pt}
\begin{prop}
    For any types $s,t\in T$, we have
    \[\Delta_{s,t}(i)=\Delta_{s,t}\bigl(\vec z(i)\bigr)\] as defined above.
   
\end{prop}

\begin{proof} Fix types $s, t$. We want to compute the number of vertices of type $t$ created or destroyed by activating a vertex of type $s$. The starting vertex is always colored, so reduces the number of vertices of its own type by one. Each branch away from the starting vertex in the remainder behaves independently, and there are $\deg(s)$ many such branches. So if $\Delta_B$ is the number of vertices of type $t$ created along a single branch
\[\Delta_{s,t}(i)=-\delta_{s,t}+\deg(s)\Delta_B.\] We just need to check that $\Delta_B=\Delta_{B,t}(\vec z(i))$ as defined above.

Let $C$ be the set of vertices in a single branch away from the origin whose parents are colored in the cascade. Then, the set of vertices in $C$ at distance $k$ from the origin is a simple branching process with growth rate $g$. In expectation $|C|=\frac{1}{1-g}$. Each vertex in $C$ reduces the number of vertices of its own type by $1$. And, a vertex of type $(d,c)$ in $C$ creates a vertex of type $(d-1, c-1)$ if its parent was colored by one of its $c$ available colors (provided $c\not=2$), or type $(d-1,c)$ if its parent was colored by one of the $p-c$ other colors. So, the total contribution to vertices of type $t=(d,c)\in T$ is
\[\frac{1}{1-g}\bigl(-q_t(i)+\frac{c+1}{p}q_{(d+1,c+1)}+\frac{p-c}{p}q_{(d+1,c)}\bigr)\] where $q_u=0$ if $u\not\in T$.
     
\end{proof}

Note that these difference equations look nearly identical to the definition of the Euler approximation to a certain differential equation. (Ignoring the error term and rescaling the input). The next lemma says the trajectory of our algorithm also approximates the solution to this differential equation.

\begin{lem}[Trajectory Lemma]\label{lem:trajectory}
    Let $\vec{\sigma}$ be the solution to the differential equation
    
     \[\sigma_t(0)=\left\{\begin{array}{ll} 1 & t=(4,3) \\ 0 & \mbox{otherwise} \end{array}\right.\]
    and, for all $t\in T$,
    \[\frac{d}{dx} \sigma_t(x)=\sum_{s\in T} p_s\Delta_{s,t}(\vec\sigma).\]
    
    Then, for any $\delta>0$ and $R\in [0,\infty)$ so that $\vec \sigma(x)$ is subcritical for $x\leq R$, and for all small enough $\epsilon$,
    \[d(\vec z(n),\vec \sigma(\epsilon n))\leq \delta\] for all $\epsilon n< R$
   
\end{lem}
\begin{proof}
    First, define $\vec y:\N\rightarrow \R^T$ by the difference equation
    \[y_t(0)=\left\{\begin{array}{ll} 1 & t=(4,3) \\ 0 & \mbox{otherwise} \end{array}\right.\]
    \[y_t(i+1)=y_t(i)+\epsilon \sum_{s\in T} p_s\Delta_{s,t}(\vec y(i)).\] That is $y$ is just the Euler approximation to $\sigma$ with step $\epsilon$. Since $\Delta_{s,t}$ is Lipschitz when $y$ is kept subcritical, standard results on Euler approximation tell us that for all $R$ and all small enough $\epsilon$, $|\vec y(n)-\vec\sigma(n\epsilon)|<\delta$ for $n\epsilon<R$.

    Now, we want to show that $\vec z$ stays close to $\vec y$. The only subtlety is that the error term in Lemma \ref{lem:steplemma} depends on the expected growth of cascades in $\vec z$. 

    Fix $\epsilon$ small enough and $c>0$ so that for all $n<R/\epsilon$, $g\bigl(y(n)\bigr)<1-c$. Let $K(\vec z)$ be the implicit constant in the error term in Lemma \ref{lem:steplemma}, and note that $K(\vec z)$ depends monotonically on $g(\vec z)$. And, let $K$ be the implicit constant given growth rate $1-c/2$. We will prove by induction on $n$ that there is some constant $C$ so that for $\epsilon$ small enough, $(1)$ $|\vec y(n)-\vec z(n)|\leq(1+C\epsilon)^{n}\epsilon$, and $(2)$ $K(\vec z)\leq K$.

    The base case is trivial: $\vec z(0)=\vec y(0)$. For the induction step, let $L$ be a Lipschitz constant for $\Delta_{st}$ restricted to vectors with expected growth of cascades at most $1-c/3$ and let $C=L+K$. We compute
    \begin{align*}|\vec y_t(i+1)-\vec z_t(i+1)|&\leq |\vec y_t(i)-\vec z_t(i)|+\epsilon\sum_{s\in T}p_s|\Delta_{st}(\vec z(i))-\Delta_{st}(\vec y(i))| +K(\vec z)\epsilon^2\\
    \; & \leq \epsilon(1+C\epsilon)^i+\epsilon L\epsilon(1+C\epsilon)^i+K\epsilon^2 (1+C\epsilon)^i \\
    \; &\leq \epsilon(1+C\epsilon)^i \bigl(1+\epsilon L+\epsilon K\bigr) \\ \; & \leq \epsilon(1+C\epsilon)^{i+1}.\end{align*}
    This proves $(1)$. To see that $(2)$ holds for small enough $\epsilon$ (independent of $n$), note that for $n<R/\epsilon$, $(1+C\epsilon)^{n}$ is uniformly bounded above by $e^{CR}$. So, as long as $\epsilon$ is small enough, $\vec z$ stays close enough to $\vec y$ to ensure that $K(\vec z(i+1))<K$.
\end{proof}

More generally, if some error term is bounded for all sequences sufficiently close to the solution of a difference equation, and sequences satisfying the difference equation perturbed by this error term stay sufficiently close to the actual solution, then we can find a perturbed sequence that stays close to the actual solution. We will use this idea again in the next section, but we won't belabour the details of the induction.

Now, we need to make sure that throughout the runtime of our algorithm, $\vec z$ stays subcritical. By the above lemma, it's enough to know that the solution to the differential equation above stays subcritical. Then, taking $\epsilon$ small enough, the trajectory of $\vec z$ will stay close enough to $\vec \sigma$ to ensure the expected growth of cascades stays bounded away from 1. And here is where we need to tune our parameters carefully. The following can be checked using most desktop computers and standard numerical computing software.

\begin{prop}[Numerical computation] \label{prop:comp}
    For $(r,p)=(4,3)$ or $(6,4)$, and for tuning parameters 
    \[p_{(d,c)}=\left\{\begin{array}{cl}2^{2-2d} & d\not=1\\2^{-10} & d=1
    \end{array}\right.,\] $\sigma$ stays subcritical until the growth rate for components in the remainder is less than $1$. That is, there is some $R\in [0,\infty]$ so that $g(i)<.99999 $ for $y\leq R$, and
    \[ 
    \sum_{s\in T} (\deg(s)-1)q_s(\sigma(R))<.99999  \tag{$\star$}
    \]
\end{prop}

Now we are in position to get an approximate coloring of $T_r$ for $r=4$ or $6$. Recall that $p^+=\{0,...,p-1, \rd\}$.

\begin{thm}
    For any $\epsilon$, there is an fiid $3^+$-labelling of $T_4$, $\bf f$, so that $\bf f$ is proper at any vertex which is not colored red and
    \[\bb P\bigl(\bf f(v)=\rd\bigr)\leq \epsilon.\] Similarly, there is an fiid $4^+$ labelling of $T_6$ which proper at any vertex which isn't red and which assigns red with probability at most $\epsilon$.
\end{thm}
\begin{proof}
    Fix tuning parameters $p_{(d,c)}$ and time $R$ as in the previous proposition. Let $\epsilon$ be small enough that for $i<R/\epsilon$,  $\vec z(i)$ stays close enough to $\sigma(\epsilon i)$ that \[\bigl|g(i)-g\bigl(\sigma(\epsilon i)\bigr)\bigr|, \sum_{t\in T}\left|\bigl((\deg(t)-1)q_t(i)-q_{t}(\sigma(\epsilon i)\bigr)\right|<\frac{1}{r^2 10^{100}}.\] This is possible by the Trajectory Lemma, Lemma \ref{lem:trajectory}, since $\sigma$ is subcritical before time $R$. Let $\bf {\tilde f}$ be $\bf f_{\lceil R/\epsilon\rceil}$. Then, by Proposition \ref{prop:components}, the size component of a vertex $v$ in $R_{\bf{\tilde f}}$ is given by a branching process with growth 
    \[\sum_{s\in T} (\deg(s)-1)q_{s}(i)<1.\] So, almost surely, every component of the remainder is finite. Since each remaining vertex has at least $2$-colors available from $p$, we can use Luzin-Novikov to select a $p$-coloring on each of the remaining finite components.

 In each stage $i$ of the greedy coloring process $v$ is colored red with probability at most $c\epsilon^2 \bb P(v\in R_{\bf f_i})$ for some $c$. Since in each stage, the probability that $v$ is colored something other than red is at least $c'\epsilon \bb P(v\in R_{\bf f_i})$, we have that 
    \[\bb P(\bf {\tilde f}(v)=\mathrm{red})\leq \frac{c}{c'} \epsilon.\]
    Thus by taking $\epsilon$ small enough we have found an approximate coloring.

\end{proof}

\section{Tidying up}

The goal of this section is to adapt the preceding argument to get proper colorings on all of $T_4$ and $T_6$ with control over the size of color classes. That is, we will prove Theorems \ref{thm:mainthm} and \ref{thm:submainthm}.

Let's first consider just the case of $T_6.$ We want to take the coloring $\bf f$ described in the previous section and somehow modify it to clean up any adjacent red vertices (recall this was a $4^+$-labelling). Suppose we delete all the colors in $\bf f$ on red vertices and their neighbors. Every uncolored vertex now has at least two colors from $5$ available to it, so it suffices to show the components of the uncolored graph are finite. This can probably be done, but the analysis is a little easier if we instead modify the coloring above to ensure this happens.

We will color very much as in the greedy coloring process, but in between rounds of coloring we will add a buffer around the red vertices. This may start secondary cascades, which may introduce new red vertices, and so on. We need to check that that this buffer can be added in a local way (i.e. that the induced component of uncolored vertices within three steps of a red vertex are finite), that the induced components of red vertices added in a single coloring step are finite, and that the total probability of coloring a vertex in this buffer step is at most $O(\epsilon^2)$. This can all be done by bounding the relevant quantities in by simple branching processes.

First we give a formal definition of our modified coloring process.

\begin{dfn}
    The \textbf{modified coloring process} for approximately $p$-coloring $T_r$ in defined inductively in steps. Each step is defined inductively in $\omega$ rounds. The modified coloring process at step $i$ and round $j$, $\bf f_{ij}$ is defined as follows:
        \begin{itemize}
            \item $\bf f_{i+1,0}$ is gotten from $f_{i,\omega}$ by the same rule as the greedy coloring process: activate vertices independently with weight depending on their type and color in the cascades around them
            \item To get $\bf f_{i,j+1}$, pick some proper $p$-coloring of the components of \[B_3(\bf f_{ij}\inv(\rd))\cap R_{\bf f_{ij}}\] (assuming these are finite), and then color in the cascades at the boundaries of these components following rules $(2-4)$ or Definition \ref{dfn:greedycolor}
            \item Finally, let $\bf f_{i\omega}=\bigcup_{j} \bf f_{i,j}$
        \end{itemize} 
\end{dfn}

Now we go about checking that, provided $\epsilon$ is small enough, $(1)$ the components we color in the second item are finite, $(2)$ the total change made in all rounds during step $i$ is small enough to not change the analysis from the preceding section, and $(3)$ the induced components on uncolored vertices in $B_3\bigl(\bf f\inv(\rd)\bigr)$ are finite. Note that, by the exact same argument as before, $\bf f_{ij}$ is nearsighted for all $i,j$.

\begin{dfn}
    Say that a vertex is \textbf{active} in $\bf f_{ij}$ with $j>0$ if it is the start of a cascade. That is, $v$ is active if $d\bigl(v, \bf f_{ij}\inv(\rd)\bigr)= 3$ and coloring $v$ with the rest of its component reduces the colors available to one of $v$'s neighbors to $1$. 
\end{dfn}

By the comments following Lemma \ref{lem:trajectory}, in proving the distribution of types in $\bf f_{ij}$ stays close to $\sigma(i\epsilon)$, we may assume that $\bf f_{ij}$ stays uniformly subcritical. The following is a simple computation:
\begin{prop}
    Provided $\bf f_{ij}$ stays subcritical, there is some $K$ so that for all $i,j$ with $j>0$,
    \[\bb P(v\mbox{ is active in }\bf f_{ij+1})<r^3 \,\bb P(v\mbox{ colored red first by }\bf f_{ij})\]
    and
    \[\bb P(v \mbox{ is colored red first by }\bf f_{ij})< {r \choose 2} \bigl(K\, \bb P(v \mbox{ is active in }\bf f_{ij})\bigr)^2\]
\end{prop}
Roughly, $K$ is the expected size of a cascade. The argument is basically the same as in Lemma \ref{lem:step}. Applying these propositions inductively we get
\begin{cor} \label{cor:uglybound}
    For $j>1$, provided $\bf f_{ij}$ stays subcritical,
    \[\bb P(v\mbox{ is colored red first by }\bf f_{ij})\leq (K^2r^5)^{ (2^{j+1}-1)} \epsilon^{2^j}\]
\end{cor} 

With this in hand, we can bound the components in $B_3\bigl(\bf f_{ij}\inv(\rd)\bigr)\cap R_{\bf f_{ij}}$:

\begin{lem}Provided $\bf f_{ij}$ stays subcritical and $\epsilon$ is small enough, the induced components in $B_3\bigl(\bf f_{ij} \inv (\rd)\bigr)\cap R_{\bf f_{ij}}$ are almost surely finite for all $i,j$. 
\end{lem}
\begin{proof}
    Note that, by construction, if $v$ remains at step $i$ and round $j$ and is within $3$ steps of a red vertex, then that red vertex must have been first colored in the previous round. 

    Fix $i,j$ and consider the process of growing a subgraph $G$ from a vertex $v_0$ by first putting $v_0$ in $G$ if $v_0\in B_3(\bf f_{ij}\inv(\rd))\cap R_{\bf f_{ij}}$ and then iteratively applying the following rule: put $u$ and all the vertices in the path from $u$ to $G$ in $G$ if $d(u,G)\leq 3$ and there is a path from $u$ to a red vertex of length at most three which avoids $G$. 
    
    The number of vertices in $G$ added in the $n^{th}$ iteration form a simple branching process with growth rate $O\bigl(\bb P(v\mbox{ is colored red first in }\bf f_{ij-1})\bigr)$ with a constant depending only on $r$, and the component of $v_0$ in $B_3\bigl(\bf f_{ij}\inv(\rd)\bigr)\cap R_{\bf f_{ij}}$ is contained in $B_3(G)$. 

    By the computation above, for $\epsilon$ small enough, this growth rate is smaller than $1$, so $G$ is almost surely finite, as desired.
\end{proof}

And, we can bound the red components added in a single round:

\begin{lem} \label{lem:redbd} Provided $\epsilon$ is small enough and $\bf f_{ij}$ stays subcritical, the red components added in each round are almost surely finite.
\end{lem}
\begin{proof}
    Note that if $v$ is colored red in step $i$ and round $j$, then at least two of its neighbors are colored in the same round by cascades starting from different active vertices (or are colored red at the same time as $v$). Call a neighbor $u$ of $v$ a red witness for $v$ if there is an active vertex $u'$ on the branch $\br_v(u)$ so the cascade from $u'$ reaches $B_1(u)$.
    
    Fix a vertex $v_0$ and consider the subgraph $G$ grown according to the following rules: add $v_0$ to $G$ if $v_0$ is colored red in step $i$ and round $j$; iteratively add $u$ to $G$ if $u$ is adjacent to a vertex $v$ in $G$ and each of $v$ and $u$ can find a pair of red witnesses not containing the other vertex. 
    The number of vertices at distance $n$ from $v_0$ is a simple branching process with growth rate bounded by $O\bigl(\bb P(v\mbox{ is active in }\mathbf f_{ij})\bigr)$. For epsilon small enough, this growth late is less than $1$.
\end{proof}

Putting this all together, we can prove:

\begin{thm} For any $\epsilon>0$, there is an fiid $5$-coloring of $T_6$ so that any vertex receives color $4$ with probability at most $\epsilon$, and similarly there are $4$- and $5$-colorings of $T_4$ and $T_5$ which are approximate $3$ and $4$ colorings respectively. \end{thm}
\begin{proof} We will handle the case of $T_6$, the other cases are just changes of indices.

We first show using the above lemmas that there is an fiid $4^+$-labelling of $T_6$, $\bf f$ so that 
    \begin{enumerate}
        \item $\bb P\bigl(\bf f(v)=\mathrm{red}\bigr)<\epsilon$
        \item almost surely the component of $v$ in the induced subgraph on $B_1\bigl(f\inv({\mathrm{red}})\bigr)$ is finite
        \item If $\bf f(v)$ is not red, then $\bf f(u)\not=\bf f(v)$ for any neighbor $u$ of $v$.
    \end{enumerate}

Let $\bf f_i=\bf f_{i\omega}$. Then provided $\bf f_i$ stays subcritical, $(1)$ holds by the same argument as in the previous section, $(2)$ holds of each $\bf f_i$ by Lemma \ref{lem:redbd}, and $(3)$ holds by construction. By Corollary \ref{cor:uglybound}, the probability a vertex is colored in $\bf f_i$ and not in $\bf f_{i0}$ is at most
\[\sum_{j>1} (K^2r^5)^{2^{j+1}-1}\epsilon^{2^j}\leq O(\epsilon^2),\]
and the change between $\bf f_i$ and $\bf f_{i+1,0}$ is given by the same formula as in the Step Lemma, Lemma \ref{lem:steplemma}.

Now, the distribution of types in $\bf f_i$ approximates the same $\vec \sigma$ as in the Trajectory Lemma, Lemma \ref{lem:trajectory}. That is, for $R$ as in Proposition \ref{prop:comp} and $\delta>0$, there is $\epsilon$ small enough that for any $i<R/\epsilon$\[\bigl|\bb P(\tau_{\bf f_i}(v)=t)-\sigma(\epsilon\,i)\bigr|<\delta.\] In particular, $\bf f_i$ stays subcritical and $\bf f_{\lceil R/\epsilon\rceil}$ has finite components in the remainder. We can again use Luzin--Novikov fill this in to a coloring with properties $(1-3)$ as desired.

The last step is to modify our approximate $4^+$ coloring $\bf f$ to get a proper $5$-coloring. If we erase all the colors on $B_1\bigl(f\inv(\mathrm{red})\bigr)$, then the uncolored subgraph has finite components by $(2)$; we have a partial proper coloring by $(3)$; and a given vertex is uncolored with probability at most $(r+1)\epsilon$ by $(1)$. The uncolored vertices all have 2 colors from $5$ available to them: either they were not red before in which case they have their original color or $5$ available, or they were red in which case they don't see any colored vertices. By list 2-colorability of trees, we can choose a coloring on each finite component using Luzin--Novikov.\end{proof}

This makes $3$ the smallest known rank of a free group so that the measurable chromatic number of $F_n$ is less than $2n$.

We might try to modify the coloring further to get a 3-coloring of $T_4$. For instance, we try to resolve all conflicts whenever to cascades class. But this will cause a much larger number of secondary cascades. A careful observer will note that you only need to change colors on the smaller of the two cascades involved in a conflict, but it's unclear if this is enough to guarantee that few enough vertices are colored to salvage the analysis.

\section{Some open questions}

The most obvious open problem is:
\begin{prb}
    Does $F_2$ have a measurable $3$-coloring?
\end{prb} It is unclear whether the present method is enough to answer this question.

The proof given above gives a uniform local algorithm for coloring with uniform complexity $O\bigl(\log(\epsilon\inv)\bigr)$, and only yields proper colorings for the Bernoulli graphing of $T_n$. It seems that this is the best we can do using this simple greedy algorithm.

\begin{prb}\label{q:msrbl}
    Does every $6$-regular treeing admit a measurable $5$-coloring?
\end{prb}
\begin{prb}\label{q:algorithm}
    Is there a $o\bigl(\log(\epsilon\inv)\bigr)$ uniform algorithm for $5$-coloring $T_6$?
\end{prb}

The approximate versions of these questions are easily answered in the affirmative by standard results on weak containment and approximation by block factors. Of course, we could ask similar questions about $3$-coloring $T_4$ or $4$-coloring $T_6$, but these are even further out of reach. And by recent work connecting measurable combinatorics and distributed computing \cite{localtrees}, a negative answer to Problem \ref{q:msrbl} implies a negative answer to Problem \ref{q:algorithm}.

\end{document}